\documentclass[12pt]{article}
\usepackage{graphicx} % Required for inserting images
\usepackage[margin=1in]{geometry}
\usepackage{amsfonts, mathrsfs, mathtools}
\usepackage{amsthm}
\usepackage{amsmath}
\usepackage{amssymb}
\usepackage{commath}
\usepackage{tikz}
\usetikzlibrary{matrix}
\usepackage{tikz-cd}
\usepackage{hyperref}
\usepackage{enumerate}
\usepackage{fullpage}
\usepackage{indentfirst}

\def\Q{\mathbb Q}
\def\R{\mathbb R}

\def\c{\mathfrak{c}}

\def\E{\mathcal{O}_E}

\def\O{\text{O}}

\def\a{\mathfrak{a}}

\numberwithin{equation}{section}
\newtheorem{thm}[equation]{Theorem}
\newtheorem{conj}[equation]{Conjecture}
\newtheorem{lem}[equation]{Lemma}
\newtheorem{defn}[equation]{Definition}
\newtheorem{rmk}[equation]{Remark}

\newcommand{\gal}[1]{\mathrm{Gal}( #1)}
\newcommand{\D}[2]{\mathrm{Disc}(#1 / #2 )}
\newcommand{\ds}[1]{\mathrm{Disc}( #1 )}
\newcommand{\n}[3]{\mathrm{Nm}_{#2/ #3}(\mathrm{Disc}({#1/ #2})}
\newcommand{\rcg}[2]{\mathrm{Cl}_{#1} [#2] }

\newcommand{\z}[4]{\#\mathcal{F}_{#4} (#1, \, #2 , \, #3)}
\newcommand{\wt}[4]{\mathcal{F}_{#4} (#1, \, #2 , \, #3)}
\newcommand{\y}[3]{\mathcal{F}_{#3} (#1, \, #2 )}

\newcommand{\ind}[1]{\mathrm{\mathrm{ind}}( #1)}

\allowdisplaybreaks

\title{A power-saving error term in counting $C_2 \wr H$ extensions of an arbitrary base field parametrized by discriminants}
\author{Arijit Chakraborty}

\begin{document}

\maketitle

\begin{abstract}
We study Malle's conjecture for the group $C_2 \wr H$ where $H$ is a permutation group. Malle's conjecture for this case was proved by J{\"u}rgen Kl{\"u}ners in \cite{kluners} under mild conditions for $H$. In this article, we provide an alternative method to obtain the explicit main term and a power-saving error term for $C_2 \wr H$ extensions of an arbitrary number field. Furthermore, our method allows us to relax the assumptions for $H.$
\end{abstract}

\section{Introduction}
One of the most important problems in arithmetic statistics is to count the number of number field extensions of an arbitrary base field with bounded discriminant, and the Galois group being isomorphic to a fixed permutation group $G$. To be precise, we ask the following question:
\par \textbf{Question.} Fix a number field $F$ and an algebraic closure $\overline{\Q}$ of $\Q$ containing $F.$ For a field extension $K/F$ of degree $n$, let $\widetilde{K}$ denote the Galois closure of $K/F.$ The Galois group $\gal{\widetilde{K}/F}$ together with the action of permuting the $n$ embeddings of $K$ into $\overline{\Q}$ can be viewed as a transitive permutation subgroup of $S_n.$ Suppose $G$ is a transitive permutation group of order $n.$ Fix an embedding $G \hookrightarrow S_n.$ We introduce the following notations:
\begin{align*}
\begin{split}
\y{F}{G}{n} \coloneqq  \{K/F  \mid [K:F]=n, \, \gal{\widetilde{K}/F} \cong G \}
\end{split}
\end{align*}
and
\begin{align*}
\begin{split}
\mathcal{F}_n(F, \, G, \, X ) \coloneqq \{K/F \mid [K:F]=n, \, \gal{\widetilde{K}/F} \cong G, \,|\n{K}{F}{\Q}| \leq X\}.
\end{split}
\end{align*}
 How does $\z{F}{G}{X}{n}$ grow as $X \to \infty$?

In this article, we will focus on the case where $G \cong C_2 \wr H$ for any transitive permutation group $H$. In 2012, J{\"u}rgen Kl{\"u}ners \cite{kluners} proved that the counting function $\z{F}{C_2 \wr H }{X}{n}$ is asymptotic to $c(F, C_2 \wr H)X$ for some positive constant $c(F, C_2 \wr H) >0$ assuming mild conditions on $H$. We restate the main result of \cite{kluners} here.

\begin{thm}[Klüners, 2012]\label{thm:kluners main theorem}
Let $H$ is a transitive subgroup of $S_n$ and fix an embedding $H \hookrightarrow S_n.$ Assume that there exists at least one extension $E$  of $F$ of degree $n$ such that $\gal{\widetilde{E}/F} \cong H$. Furthermore, assume that the following estimate holds
\begin{align*}
    \z{F}{H}{X}{n} = \O_{F,H,\epsilon}(X^{1+\epsilon}).
\end{align*}
Then,
\begin{align*}
    \z{F}{C_2 \wr H}{X}{2n} \sim c(F,C_2 \wr H)X
\end{align*}
where the constant $c(F,H)$ is given by 
\begin{align}\label{eq:leadconst}
   c(F, C_2 \wr H)=\frac{1}{\alpha(H)}  \sum_{K \in \y{F}{H}{n}} \frac{1}{2^{i(K)}}\cdot \frac{\mathrm{Res}_{s=1}\zeta_K(s)}{\zeta_K(2)}\cdot \frac{1}{|\n{K}{F}{\Q}|^2},
\end{align}
where $\zeta_K(s)$ denotes the Dedekind zeta function associated with $K$, $i(K)$ is the number of pairs of complex conjugate embeddings of $K$ in $\overline{\Q}$, and $\alpha(H)$ is a positive integer depending only on $H$.
\end{thm}
This result was subsequently strengthened by the work of Brandon Alberts, Robert J. Lemke Oliver, Jiuya Wang, and Melanie Matchett Wood \cite{inductivecounting}. In particular, Corollary 1.4 of \cite{inductivecounting} relaxes the hypothesis on $H$ to $\z{F}{G}{X}{n} \ll X^{\frac{3}{2}-\delta}$ for some $\delta >0.$

Once we know the asymptotic behavior of the counting function $\z{F}{G}{X}{n}$, a natural question arises: How rapidly does this counting function $\z{F}{G}{X}{n}$ approach the asymptotic? In other words, can we bound the difference between the actual count and the asymptotic? We refer to this difference as the error term. The method used by Klüners can do slightly better than the main term. Even though the quality of the error term is not mentioned explicitly, an application of a suitable Tauberian theorem would yield an error term of $\O(X^{\frac{5}{6}+\epsilon})$. In this article, we provide an improved error term for this counting function. 

Prior to this, in the case $H=C_2$(so that $G \cong C_2 \wr C_2 \cong D_4$),  Henri Cohen, Francisco Diaz Y Diaz, and Michel Olivier \cite{cohend4}  proved the asymptotic formula for the counting function $\z{F}{D_4}{X}{4}$ obtaining a power saving error term of $\O(X^{\frac{3}{4}+\epsilon}).$

\begin{thm}[Cohen, Diaz Y Diaz, Olivier; 2001]\label{thm:cohen D4 }
The number $\wt{\Q}{D_4}{X}{4}$ of quartic $D_4$ extensions of $\Q$ up to isomorphism is asymptotic to $c(\Q, D_4)X$, i.e.,
  $$\z{\Q}{D_4}{X}{4} = c( \Q, D_4) X+ \O(X^{\frac{3}{4}+\epsilon}),$$     
with the constant $c(\Q, D_4)$ given as follows:
$$c(\Q, D_4)= \frac{1}{2}\sum_{K \in \y{\Q}{C_2}{2}}  \frac{2^{-i(K)}}{\ds{K}^2} \cdot\frac{\mathrm{Res}_{s=1}\zeta_K(s)}{\zeta_K(2)},$$
where $\zeta_K(s)$ denotes the Dedekind zeta function associated with $K$, and $i(K)$ is the number of pairs of complex conjugate embeddings of $K$ in $\overline{\Q}$. 
\end{thm} 

However, in the quartic $D_4$ case, assuming the Generalized Riemann Hypothesis, one gets an error term of $\O(X^{\frac{1}{2}+\epsilon})$. To this end, in 2024 , Kevin J. McGown and Amanda Tucker \cite{kevinamanda} improved the error term in the count of quartic $D_4$ field extensions without assuming GRH. They proved the following theorem:

\begin{thm}[McGown, Tucker; 2024]\label{thm:D4 error term}
\begin{align*}
\z{\Q}{D_4}{X}{4}  = & c(\Q, D_4)X +\O(X^{\frac{5}{8}+\epsilon}).
\end{align*}
\end{thm}   

In this article, we extend the work of McGown and Tucker to obtain a better power-saving error term in counting $C_2 \wr H$ field extensions over an arbitrary base field $F.$ The following is our main theorem that gives power saving error term for the counting function $\z{F}{C_2 \wr H}{X}{2n}$:

\begin{thm}\label{thm:main theorem} 
Let $F$ be a number field of degree $m$. Let $H$ be a transitive subgroup of $S_n$ and fix an embedding $H \hookrightarrow S_n.$ Assume that there exists at least one extension $E$  of $F$ of degree $n$ such that  $\gal{\widetilde{E}/F} \cong H$. Furthermore, assume 
\begin{equation*}
    \z{F}{H}{X}{n} \ll X^a,
\end{equation*}
where $a \leq \frac{3}{2}$ is a constant. 
\begin{enumerate}[(i)]
    \item If $a \leq \frac{3}{2}-\frac{9mn-1}{2mn(mn+1)}$ then we have
\begin{equation*}
    \z{F}{C_2 \wr H}{X}{2n} =c(F, C_2 \wr H)X +\O(X^{1-\frac{2}{mn+1}+\epsilon}) +\O(X^{\frac{3}{4}+\epsilon}).
\end{equation*}
\item If $\frac{3}{2}-\frac{9mn-1}{2mn(mn+1)} < a \leq \frac{3}{2}$ then we have
\begin{align*}
    \z{F}{C_2 \wr H}{X}{2n} =c(F, C_2 \wr H)X +  \O(X^{\frac{2}{5}(1+a-\frac{1}{2mn})+\epsilon})+\O(X^{\frac{3}{4}+\epsilon}).
\end{align*}
\end{enumerate}
\end{thm}

\begin{rmk}
\begin{enumerate}[(i)]
\item Our result proves Conjecture \ref{conj1} with a weaker hypothesis on $H$ than Corollary 1.4, \cite{inductivecounting}. It only requires $\z{F}{H}{X}{n} \ll X^\frac{3}{2}$, instead of a power savings from this.
\item This result provides a better power-saving error term than Klüners's work \cite{kluners}. For a given base field $F$ and a given permutation group $H$, the best possible scenario is case (i) of this theorem, i.e, when the bound for $\wt{F}{H}{X}{n} \ll X^a$ is such that $ a \leq \frac{3}{2}-\frac{9mn-1}{2mn(mn+1)}$. If the upper bound for $\wt{F}{H}{X}{n}$ is worse, then we will be reduced to case (ii). As a result, it fails to provide a better power-saving error term than the one expected from \cite{kluners} for $mn \geq 11.$
\end{enumerate}
\end{rmk}

We provide a recipe to obtain even better power-saving error term for $H$ being certain small groups. We demonstrate this by considering the special case $H\cong S_3 $ viewed as a transitive subgroup of $S_6$ via the regular representation. We prove the following result:
\begin{thm}\label{thm:main theorem S3}
We have 
\begin{equation*}
    \z{\Q}{C_2 \wr S_3}{X}{12} = c(\Q, C_2 \wr S_3)X + \O(X^{\frac{5}{7}+\epsilon} )
\end{equation*}
where $c(F, C_2 \wr S_3)$ is given by equation \eqref{eq:leadconst}.
\end{thm}
For $H \cong C_3 $ viewed as a subgroup of $S_3$ via the regular representation, we obtained an even better power-saving error term using better bounds for the size of the $2$-torsion subgroup of the class group. In this case, we prove the following theorem:
\begin{thm}\label{thm:main theorem C3}
    We have 

    \begin{equation*}
        \z{\Q}{C_2 \wr C_3}{X}{6}= c(\Q, C_2 \wr C_3)X + \O ( X^{\frac{2(1+ \delta  )}{5}+\epsilon}),
    \end{equation*}
where $\delta$ is such that $|\rcg{K}{2}| \ll |\ds{K}|^{\delta +\epsilon}$ for each $K \in \y{\Q}{C_3}{3}$. The value of $c(\Q, C_2 \wr C_3)$ is given as follows: 
$$c(\Q, C_2\wr C_3)=\sum_{K \in \y{\Q}{C_3}{3}}\frac{1}{2^{i(K)}}\cdot \frac{\mathrm{Res}_{s=1}\zeta_K(s)}{\zeta_K(2)}\cdot \frac{1}{|\ds{K}|^2}.$$ 

\end{thm}

\begin{rmk}
Using the value of $\delta=0.2784\dots$ given in \cite{bhargava2017bounds2torsionclassgroups}, we get the error term in Theorem \ref{thm:main theorem C3} for the group $C_2 \wr C_3$ is  $\O(X^{0.51136\dots+\epsilon})$, which is much better than the expected error term of $\O(X^{\frac{5}{6}+\epsilon})$ from \cite{kluners}, and it is also much better than the error term $\O(X^{\frac{3}{4}+\epsilon})$ provided by Cohen, Diaz Y Diaz, and Olivier in \cite{cohend4}
for the quartic $D_4(\cong C_2 \wr C_2)$ case. Additionally, if one assumes the $2$-torsion Conjecture, i.e, $\#\rcg{K}{2} \ll |\ds{K}|^{\epsilon}$ for $\epsilon$ arbitrarily small, it would yield an improved error term of $\O(X^{\frac{1}{2}+\epsilon})$.
\end{rmk}

\subsection{Brief historical review: Malle's conjecture}
In \cite{Malle1}, \cite{Malle2}, Gunter Malle made a general prediction for the asymptotic growth of the counting function $\z{F}{G}{X}{n}$, which leads to the following conjecture

\begin{conj}\label{conj1}
Let $F$ be a number field and $G$ is a transitive subgroup of $S_n.$ Then, there exist positive constants  $a, \, b, \,c$ depending on $F$, and $G$ such that 
\begin{equation*}
    \z{F}{G}{X}{n} \sim c X^{a} (\log(X))^{b}.
\end{equation*}

\end{conj}
We introduce the following definition
\begin{defn}
  Let $G \leq S_n$ be a transitive subgroup acting on the set $\Omega =\{1,2, \dots, n\}.$
  \begin{enumerate}
      \item For $g \in G$ we define the index of $g$ to be $$\ind{g}\coloneqq n-\text{number of orbits of }g \text{ in } \Omega.$$ 
      \item $\ind{G} \coloneqq \min\{\ind{g}:1\neq g \in G\}.$
      \item $a(G) \coloneqq \ind{G}^{-1}.$
  \end{enumerate}
\end{defn}

The constant $a(G)$ is known as Malle's constant. Malle's predicted value of the constant $a$ in Conjecture \ref{conj1} is this $a(G).$  Malle also proposed a value of the constant $b$. However, Jürgen Klüners \cite{kluners2005counter} produced a counterexample demonstrating that the proposed value of $b$ cannot hold in general. We also remark that Conjecture \ref{conj1} was known in several cases prior to Malle's prediction, most notably, David J. Wright \cite{wrightdensity} established Conjecture \ref{conj1} for all abelian groups. For a comprehensive historical  review of the number field counting problem, we refer the reader to  \cite{alinad4}. 

\section{Notations}
This section lists notations and conventions that we will use throughout this article. 

For any group $A$ and a transitive permutation group $B$ with a fixed embedding $B \hookrightarrow S_m$, we denote $A \wr B$ to be the wreath product of $A$ with $B.$

Let $E$ be a number field. $\ds{E}$ denotes the absolute discriminant of the number field $E$.  

We use $\mathrm{Cl}_E$ to denote the class group of $E$, and $\rcg{E}{m}$ denotes the $m$-torsion subgroup of the class group.

We denote by $r(E)$ the number of real embeddings of $E$ and $i(E)$ is the number of pairs of complex conjugate embeddings of $E$ in $\overline{\Q}$.

For an extension of number fields $E/F$, $\D{E}{F}$ denotes the relative discriminant ideal, and $\widetilde{E}$ denotes the Galois closure of $E/F$. 

We write $f(X) \sim g(X)$ if they satisfy the relation $\lim_{X \to \infty} \frac{f(X)}{g(X)}=1.$ Furthermore, we write $f(X) \ll g(X)$ or $f(X) = \O(g(X) )$ if they satisfy the inequality $|f(X)|  \leq C g(X) $ when $X \geq T$, for some constants $C$ and $T.$

\section{Counting \texorpdfstring{$C_2 \wr H$}{C2 wr H}-extensions of a number field}

\par Suppose the degree of the base field $F$ is given by $$m=[F: \Q] .$$ Throughout this section, we will assume that there exists at least one extension of $F$ of degree $n$ such that the Galois group of the Galois closure is isomorphic to $H$ and
\begin{equation*}
    \z{F}{H}{X}{n} \ll X^a,
\end{equation*}
where $a$ is a constant such that $a \leq \frac{3}{2}$.

If $L/F$ is an extension with Galois group $C_2 \wr H$ then there exists a subfield $K \leq L$ containing $F$ such that $\gal{L/K}=C_2$ and $\gal{K / F}=H$. Thus, we have
\begin{align}\label{thm:count quadratic over H extension}
\sum_{K \in \mathcal{F}_n(F, \, H, \, X^{1/2})} \sum_{L \in \wt{K}{C_2}{X/|\n{K}{F}{\Q}|^2}{2}} 1 =  &\alpha(H) \z{F}{C_2 \wr H}{X}{2n}  +E_{2n}(F, C_2 \wr H, X),
\end{align}
where $\alpha(H)$ denotes the multiplicity of $C_2 \wr H$-extensions of $F$ appearing in this sum and $E_{2n}(F, \, C_2 \wr H, \, X)$ is defined as follows:

\begin{align*}
E_{2n}(F, \, C_2 \wr H, \, X) \coloneqq \# \{L/K/F: [L:K]=2, & \, \gal{K/F} \cong H,\\
&\gal{\widetilde{L}/F} \neq C_2 \wr H,|{\n{L}{F}{\Q}}|\leq X \}.
\end{align*}
J{\"u}rgen Kl{\"u}ners proved(Theorem 5.2, \cite{kluners}) that 
\begin{equation}\label{eq:klunersbound}
  E_{2n}(F, \, C_2 \wr H, \, X) \ll X^{\frac{3}{4}+\epsilon}.
\end{equation}
Thus, to obtain the main asymptotic for $\z{F}{C_2 \wr H}{X}{2n}$ it suffices to analyze the left-hand side of equation \eqref{thm:count quadratic over H extension}. 

We will use Theorem 2 of \cite{kevinamanda}, which we restate here for our purpose.

\begin{thm}[McGown, Tucker; 2024]\label{thm:explicitC2}
  Let $F$ be a number field of degree $m \geq 2.$ Then, we have
\begin{align*}
\begin{split}
\sum_{K \in \wt{F}{C_2}{X}{2}} 1 =& \frac{1}{2^{i(F)}} \frac{\mathrm{Res}_{s=1}\zeta_F(s)}{\zeta_F(2)}X + \\
& \rcg{F}{2}\cdot
\begin{cases}
\O(|\ds{F}| ^{\frac{1}{3}} \log{(|\ds{F}|)X^{\frac{1}{2}} \log X} ) & \text{if } m=2 \\
\O( |\ds{F}|^{\frac{1}{4}}(\log  |\ds{F} )^2  X^{\frac{1}{2}}(\log X )^3) & \text{if } m=3 \\
\O_m( |\ds{F}^{\frac{1}{m+1}}X^{1-\frac{2}{m+1}}(\log X)^{m-1}) & \text{if } m>3 \\ 
\end{cases}
\end{split}    
\end{align*}
\end{thm}

Applying Theorem \ref{thm:explicitC2} on the left-hand side of equation $\eqref{thm:count quadratic over H extension}$, we get

\begin{align}\label{eq:separate main and error term}
& \sum_{K \in \mathcal{F}_n(F, \, H, \, X^{1/2})} \sum_{L \in \wt{K}{C_2}{X/|\n{K}{F}{\Q}|^2}{2}}  1 \nonumber \\
& = \sum_{K \in \mathcal{F}_n(F, \, H, \, X^{1/2})} \frac{1}{2^{i(K)}}\cdot \frac{\mathrm{Res}_{s=1}\zeta_K(s)}{\zeta_K(2)}\cdot \frac{X}{|\n{K}{F}{\Q}|^2} + \nonumber \\
&\sum_{K \in \mathcal{F}_n(F, \, H, \, X^{1/2})} \#\rcg{K}{2}\cdot  \O\left( |\ds{K}|^{\frac{1}{mn+1}}\cdot\left(\frac{X}{|\n{K}{F}{\Q}|^2}\right)^{1-\frac{2}{mn+1}+\epsilon}\right).
\end{align}
We analyze the main term and the error term in \eqref{eq:separate main and error term} separately. 

\subsection{Sum of main terms}
\par In this section, we study the main term in equation \eqref{eq:separate main and error term}, i.e, we analyze the sum

\begin{equation*}
   \sum_{K \in \mathcal{F}_n(F, \, H, \, X^{1/2})} \frac{1}{2^{i(K)}}\cdot \frac{\mathrm{Res}_{s=1}\zeta_K(s)}{\zeta_K(2)}\cdot \frac{X}{|\n{K}{F}{\Q}|^2}. 
\end{equation*}
We will require the following lemma:

\begin{lem}\label{lemma: main lemma}
Suppose $f$ and  $g $ are mappings from $\R \to \R$ such that $g$ is continuously differentiable with $g^{\prime}(t) \leq 0$ for all $t \geq 1$ and
\begin{align*}
& \# \{K/F \mid [K:F] =k, \gal{K/F} \cong H , |\n{K}{F}{\Q}| \leq X  \}  \ll f(X). 
\end{align*}
   Then,
\begin{align*}
&\sum_{K \in \wt{F}{H}{Z}{k} \setminus \wt{F}{H}{W}{k}} g(|\n{K}{F}{\Q}|)   \ll f(Z)|g(Z)|+ f(W)|g(W)|-\int_{W}^{Z} f(t) g^{\prime} (t) dt .
\end{align*}
In particular, if $k >1$ and $W=1$ then 
\begin{equation*}
 \sum_{K \in \wt{F}{H}{Z}{k}} g(|\n{K}{F}{\Q}|) \ll  f(Z)|g(Z)|-\int_{1}^{Z} f(t) g^{\prime} (t) dt .
\end{equation*}
\end{lem}

\begin{proof}
    We have 
    \begin{align*}
       & \sum_{K \in \wt{F}{H}{Z}{k} \setminus \wt{F}{H}{W}{k}} g(|\n{K}{F}{\Q}|) \\
       &= \sum_{W \leq l \leq Z}  \# \{K/F: [K:F]=k, \, \gal{K/F} \cong H , \, |\n{K}{F}{\Q}| = l \}\cdot g(l)\\ 
       & = \z{F}{H}{Z}{k} g(Z)-\z{F}{H}{W}{k} g(W) + \int_{W}^{Z} (-g^{\prime}(t))\cdot \#\wt{F}{H}{t}{k} dt  \\ 
       & \ll f(Z)|g(Z)|+ f(W)|g(W)|-\int_{W}^{Z} f(t) g^{\prime} (t) dt \,\,\,\,\,\,\,\, \text{[Since } g^{\prime}(t) \leq 0]
    \end{align*}

Note that for any field extension $K/F$ with degree greater than one, the absolute norm of the relative discriminant of $K/F$ cannot be equal to $1$. Thus, for $W=1$ and $k >1$, we get $\z{F}{H}{W}{k}=0$. This completes the proof of the second part of the lemma.. 

\end{proof}

Coming back to the analysis of the main term in equation \eqref{eq:separate main and error term}, observe that,  since every term in the Euler product of $\zeta_K(2)$ is greater than $1$, we have $\zeta_K(2) \geq 1$. Furthermore, from Lemma 2.1 of \cite{kluners}, we get $$\mathrm{Res}_{s=1}\zeta_K(s) \ll_n |\ds{K}|^{\epsilon}.$$ We also have the following norm-discriminant relation:
\begin{equation*}
    \ds{K}=|\n{K}{F}{\Q}|\cdot\ds{F}^n.
\end{equation*}
Therefore, $\ds{K}$ and $|\n{K}{F}{\Q}|$ differ only by a constant. Using these results and applying Lemma \ref{lemma: main lemma}, we get 

\begin{align*}
&\sum_{K \in \y{F}{H}{n} \setminus \mathcal{F}_n(F,H,X^{1/2})}\frac{1}{2^{i(K)}}\cdot \frac{\mathrm{Res}_{s=1}\zeta_K(s)}{\zeta_K(2)}\cdot \frac{X}{|\n{K}{F}{\Q}|^2}  \nonumber  \\
& \ll_{\epsilon, n} X \sum_{K \in \y{F}{H}{n} \setminus \mathcal{F}_n(F,H,X^{1/2})} \frac{|\ds{K}|^{\epsilon}}{|\n{K}{F}{\Q}|^2} \nonumber \\ 
& \ll_{\epsilon,n,F} X \sum_{K \in \y{F}{H}{n} \setminus \mathcal{F}_n(F,H,X^{1/2})} |\n{K}{F}{\Q}|^{-2+\epsilon} \nonumber \\
& \ll_{\epsilon,n,F} X( X^{\frac{a}{2}}X^{-1+\epsilon}  - (-2+\epsilon) \int_{X^{\frac{1}{2}}}^{\infty}t^a t^{-3+\epsilon} dt ) \nonumber \\
& \ll_{\epsilon,n,F} X^{\frac{a}{2}+ \epsilon}.
\end{align*}

Therefore, we get 
\begin{align}\label{eqn:mainterm}
    &  \sum_{K \in \mathcal{F}_n(F, \, H, \, X^{1/2})} \frac{1}{2^{i(K)}}\cdot \frac{\mathrm{Res}_{s=1}\zeta_K(s)}{\zeta_K(2)}\cdot \frac{X}{|\n{K}{F}{\Q}|^2} \nonumber \\
    = & \left( \sum_{K \in \mathcal{F}_n(F, \, H)} \frac{1}{2^{i(K)}}\cdot \frac{\mathrm{Res}_{s=1}\zeta_K(s)}{\zeta_K(2)}\cdot \frac{1}{|\n{K}{F}{\Q}|^2}\right) X +\O(X^{\frac{a}{2}+\epsilon})\nonumber \\
    =&c(F, C_2 \wr H)X+\O(X^{\frac{3}{4}+\epsilon}),
\end{align}
where in the last step we used the assumption $a \leq \frac{3}{2}.$
\subsection{Sum of error terms}

\parindent=10pt Now we will focus on the error term in equation \eqref{eq:separate main and error term}, i.e., we will analyze the following sum:

$$\sum_{K \in \mathcal{F}_n(F,H,X^{1/2})}  \#\rcg{K}{2}\cdot |\ds{K}|^{\frac{1}{mn+1}}\cdot\left(\frac{X}{\n{K}{F}{\Q}|^2}\right)^{1-\frac{2}{mn+1}+\epsilon} .$$

First, we will separate the sum into two parts, namely:
\begin{align*}
     \sum_{K \in \wt{F}{H}{Z}{n}}  
    +\sum_{K \in \mathcal{F}_n(F,H,X^{1/2}) \setminus \wt{F}{H}{Z}{n}}  
\end{align*}

Let us analyze the first part of this sum in detail. In 2017 \cite{bhargava2017bounds2torsionclassgroups}, Manjul Bhargava, Arul Shankar, Takashi Taniguchi, Frank Thorne, Jacob Tsimerman, and Yongqiang Zhao obtained an upper bound for the size of $2$-torsion subgroup of the class group which we will use. Let us restate the main theorem of \cite{bhargava2017bounds2torsionclassgroups} here.

\begin{thm}[Bhargava, Shankar, Taniguchi, Thorne, Tsimerman, Zhao; 2020]\label{thm:class group bound}
The size of the $2-$torsion subgroup of the class group of a number field $E$ of degree $n$ is $\O_{\epsilon}( |\ds{E}|^{\frac{1}{2}-\delta_n+\epsilon} ).$  If $E$ is a cubic or quartic field, we can take $\delta_n=\delta=0.2784\dots.$ For $n >4 $, $\delta_n=\frac{1}{2n}.$

\end{thm}

Using Theorem \ref{thm:class group bound} and Lemma \ref{lemma: main lemma}, we get 
\begin{align}\label{eq:bound first part of the error term}
& \sum_{K \in \wt{F}{H}{Z}{n}}  \#\rcg{K}{2}\cdot|\ds{K}|^{\frac{1}{mn+1}}\cdot\left(\frac{X}{|\n{K}{F}{\Q}|^2}\right)^{1-\frac{2}{mn+1}+\epsilon} \nonumber \\
&\ll_{n,F, \epsilon}  X^{1-\frac{2}{mn+1}+\epsilon} \sum_{K \in \wt{F}{H}{Z}{n}} |\n{K}{F}{\Q}|^{\frac{1}{2}(1-\frac{1}{mn})+\epsilon-2+\frac{5}{mn+1}} \nonumber \\
&\ll_{n,F, \epsilon}  X^{1-\frac{2}{mn+1}+\epsilon} \sum_{K \in \wt{F}{H}{Z}{n}} |\n{K}{F}{\Q}|^{-\frac{3}{2}-\frac{1}{2mn}+\frac{5}{mn+1}+\epsilon} \nonumber \\
& \ll_{n,F, \epsilon}  X^{1-\frac{2}{mn+1}+\epsilon}\cdot( Z^{a+\frac{5}{mn+1}-\frac{1}{2mn}-\frac{3}{2}+\epsilon}+1 ) \nonumber \\
& =X^{1-\frac{2}{mn+1}+\epsilon}\cdot(Z^{a-(\frac{3}{2} - \frac{9mn-1}{2mn(mn+1})+\epsilon}+1)
\end{align}

\begin{proof}[{Proof of Theorem} \ref{thm:main theorem}(i)]
Observe that if $a \leq \frac{3}{2}-\frac{9mn-1}{2mn(mn+1)}$ then by equation \eqref{eq:bound first part of the error term} the sum in the error term of equation \eqref{eq:separate main and error term}is bounded by $X^{1-\frac{2}{mn+1}+\epsilon}$ and this bound does not depend on the choice of $Z$. In particular, we can choose $Z=X^{\frac{1}{2}}$. Using this result and
combining equations \eqref{thm:count quadratic over H extension}, \eqref{eq:klunersbound}, \eqref{eq:separate main and error term}, and   \eqref{eqn:mainterm}, we get

\begin{align*}
\z{F}{C_2 \wr H}{X}{2n} =c(F, C_2 \wr H)X +\O(X^{1-\frac{2}{mn+1}+\epsilon}) +\O(X^{\frac{3}{4}+\epsilon}).    
\end{align*}

This completes the proof of the first part of Theorem \ref{thm:main theorem}.
\end{proof}

Now we will analyze the case $a >\frac{3}{2}-\frac{9mn-1}{2mn(mn+1)}$ in greater detail. We need to introduce the following notion:

\begin{defn}
For a number field $E$, let $V(E)$ denote the set of all $u \in E^{\times}$ such that $(u)=\a^2$ for some ideal $\a \subseteq \E.$ We define the 2-Selmer group of $E$ to be $S_2(E) \coloneqq V(E) / (E^{\times})^2.$
\end{defn}
By Lemma 3.2 of \cite{cohend4}, we have $\#\mathrm{S}_2(E) \ll \#\mathrm{Cl}_E[2]$, where the implied constant depends on the degree of the extension. Using Theorem \ref{thm:class group bound} we get
\begin{equation}\label{eqn4.11}
    \#S_2(E) \ll |\ds{E}|^{\frac{1}{2}-\delta_n + \epsilon},
\end{equation}
where $n$ is the degree of the number field $E.$
Furthermore, let $A(E)$ denote the set of squarefree ideals $\mathfrak{a}$ such that $\overline{\mathfrak{a}} \in \text{Cl}_E^2.$ Then, there exists a bijection 
$$ A(E) \times S_2(E) \to \{[L:E]=2 \} .$$ Furthermore, if $(\a, \, \overline{u})$ corresponds to the extension $L /E $ under this map then the relative discriminant of the image is given by 
\begin{equation*}
    \D{L}{E}= \frac{4\a}{\c^2},
\end{equation*}
where $\c$ is defined as follows: let $\a \mathfrak{q}^2=(\alpha_0)$ with $(\mathfrak{q}, \, 2)=1$; then $\c$ is the largest ideal such that $\c \mid 2, \, (\c, \, \a)=1$ and $x^2 \equiv \alpha_0u \, (\text{mod } \c^2)$ is solvable. We refer the reader to \cite{cohend4} for a proof of these results.

We will require Theorem 3 of \cite{lowryduda2018uniformboundslatticepoint} which counts the number of ideals in the ring of integers of a number field $E$ with absolute norm bounded by $X$:
\begin{thm}[Lowry-Duda, Taniguchi, Thorne; 2022]\label{thm:bound ideals with fixed norm}
Suppose $E$ is a number field of degree $d \geq 1.$ Then, for $X \geq 2 $ the number of integral ideals $\mathfrak{a}$ with $|\mathrm{Nm}_{E/\Q}(\mathfrak{a})| < X$ satisfies the estimate

\begin{equation*}
\# \{\mathfrak{a}: |\mathrm{Nm}_{E/\Q}(\mathfrak{a})| < X  \} = \frac{2^{r(E)}(2\pi)^{i(E)}h R}{w|\ds{E}|^{\frac{1}{2}}}X+ \O( |\ds{E}|^{\frac{1}{d+1}}X^{\frac{d-1}{d+1}}(\log X)^{d-1})
\end{equation*}
where the implied constant depends only on $d.$  
\end{thm}

\begin{proof}[Proof of Theorem \ref{thm:main theorem}(ii)]
Using Theorem \ref{thm:bound ideals with fixed norm} and Lemma \ref{lemma: main lemma}, we get the following upper bound
\begin{align}\label{eq:bound second part of the error term}
&\sum_{K \in \mathcal{F}_n(F,H,X^{1/2}) \setminus \wt{F}{H}{Z}{n}} \sum_{L \in \wt{K}{C_2}{X/|\n{K}{F}{\Q}|^2}{2}} 1   \nonumber \\
\leq & \sum_{K \in \mathcal{F}_n(F,H,X^{1/2}) \setminus \wt{F}{H}{Z}{n}} \sum_{\substack{\mathfrak{a} \in A(K) \\ |\mathrm{Nm}_{K/\Q}({\mathfrak{a})}|\leq  X/\n{K}{F}{\Q}|^2}} \sum_{\overline{u} \in S_2(K)} 1 \nonumber \\
\ll & \sum_{K \in \mathcal{F}_n(F,H,X^{1/2}) \setminus \wt{F}{H}{Z}{n}} \#S_2(K) \sum_{\substack{\mathfrak{a} \in A(K) \\ |\mathrm{Nm}_{K/\Q}({\mathfrak{a})}|\leq  X/|\n{K}{F}{\Q}|^2}} 1 \nonumber \\
\ll & \sum_{K \in \mathcal{F}_n(F,H,X^{1/2}) \setminus \wt{F}{H}{Z}{n}}  |\ds{K}|^{\frac{1}{2}-\frac{1}{2mn} +\epsilon}  \, \frac{X}{|\n{K}{F}{\Q}|^2}  \nonumber \\
\ll & X \sum_{K \in \mathcal{F}_n(F,H,X^{1/2}) \setminus \wt{F}{H}{Z}{n}} |\n{K}{F}{\Q}|^{-\frac{3}{2}-\frac{1}{2mn}+\epsilon}\nonumber\\
\ll & X Z^{-\frac{3}{2}-\frac{1}{2mn}+a+\epsilon}+X^{\frac{1}{4}-\frac{1}{4mn}+\frac{a}{2}+\epsilon}.
\end{align}

Therefore, comparing equations \eqref{eq:bound second part of the error term} and \eqref{eq:separate main and error term}, we get 
\begin{align}\label{eq:bound second part of the error term 2}
 & \sum_{K \in \mathcal{F}_n(F, \, H, \, X^{1/2}) \setminus \wt{F}{H}{Z}{n}} \#\rcg{K}{2}\cdot   |\ds{K}|^{\frac{1}{mn+1}}\cdot\left(\frac{X}{|\n{K}{F}{\Q}|^2}\right)^{1-\frac{2}{mn+1}+\epsilon}  \nonumber \\
 & \ll X Z^{-\frac{3}{2}-\frac{1}{2mn}+a+\epsilon}+X^{\frac{1}{4}-\frac{1}{4mn}+\frac{a}{2}+\epsilon}. 
\end{align}

Now, let us observe that, under our assumption $a >\frac{3}{2}-\frac{9mn-1}{2mn(mn+1)}$, the dominant term in the final expression of equation \eqref{eq:bound first part of the error term} is $X^{1-\frac{2}{mn+1}+\epsilon}Z^{a-(\frac{3}{2} - \frac{9mn-1}{2mn(mn+1})+\epsilon}.$ Thus, combining equations \eqref{eq:bound first part of the error term}, and \eqref{eq:bound second part of the error term 2} we get the following bound for the error term

\begin{align*}
& \sum_{K \in \mathcal{F}_n(F,H,X^\frac{1}{2})}  \#\rcg{K}{2}\cdot |\ds{K}|^{\frac{1}{mn+1}}\cdot\left(\frac{X}{\n{K}{F}{\Q}|^2}\right)^{1-\frac{2}{mn+1}+\epsilon} \\ \ll &   X^{1-\frac{2}{mn+1}+\epsilon}Z^{a-(\frac{3}{2} - \frac{9mn-1}{2mn(mn+1})+\epsilon}+X Z^{-\frac{3}{2}-\frac{1}{2mn}+a+\epsilon}+X^{\frac{1}{4}-\frac{1}{4mn}+\frac{a}{2}+\epsilon}. 
\end{align*}
This is optimized when we take $Z= X^{\frac{2}{5}}$, in which case we get  

\begin{align}\label{eqn:err2}
 & \sum_{K \in \mathcal{F}_n(F,H,X^{1/2)}}  \#\rcg{K}{2}\cdot |\ds{K}|^{\frac{1}{mn+1}}\cdot\left(\frac{X}{\n{K}{F}{\Q}|^2}\right)^{1-\frac{2}{mn+1}+\epsilon} \nonumber \\ \ll &  X^{ \frac{2}{5}(1+a-\frac{1}{2mn})+\epsilon} + X^{\frac{1}{4}-\frac{1}{4mn}+\frac{a}{2}+\epsilon}. 
\end{align}
Therefore, combining equations \eqref{thm:count quadratic over H extension}, \eqref{eq:klunersbound}, \eqref{eq:separate main and error term}, \eqref{eqn:mainterm}, and \eqref{eqn:err2}, we get

\begin{align}\label{eqn4.14}
   \z{F}{C_2 \wr H}{X}{2n}  = c(F, C_2 \wr H)X + \O(X^{\frac{1}{4}-\frac{1}{4mn}+\frac{a}{2}+\epsilon})+\O(X^{\frac{2}{5}(1+a-\frac{1}{2mn})+\epsilon})+\O(X^{\frac{3}{4}+\epsilon}).
\end{align}

Finally, observe that if $a < \frac{3}{2}+\frac{1}{2mn}$, then the dominant error term in equation \eqref{eqn4.14} is $X^{\frac{2}{5}(1+a-\frac{1}{2mn})+\epsilon}$. This completes the proof of Theorem \ref{thm:main theorem}.
\end{proof}

\section{Analyzing the case \texorpdfstring{$H \cong S_3$}{TEXT}}
We wish to explore the case $H \cong S_3$ in greater detail in this section. In this case, we have 
\begin{align}\label{eq:count quadratic over S3}
\sum_{K \in \mathcal{F}_6(\Q, \, S_3, \, X^{1/2})} \sum_{L \in \wt{K}{C_2}{X/|\ds{K}|^2}{2}} 1 =  & \alpha(S_3)\z{\Q}{C_2 \wr S_3}{X}{12}  +E_{12}(\Q, C_2 \wr S_3, X)
\end{align}

We intend to provide a better bound for $E_{12}(\Q, C_2 \wr S_3, X)$ than the generic one given in equation \eqref{eq:klunersbound}. We will require Theorem 3 of \cite{bhar}, which evaluates the density of Sextic-$S_3$ fields parametrized by discriminants, which we restate here:

\begin{thm}[Bhargava, Wood; 2007]\label{thm:bound for S3 fields}
We have
    \begin{equation*}
   \z{F}{S_3}{X}{6} \ll X^{\frac{1}{3}}.
   \end{equation*}

\end{thm}

First we will provide an upper bound for $E_{12}(\Q, C_2 \wr S_3, X)$. Suppose $G$ denotes a proper subgroup of $C_2 \wr S_3$. We will use Corollary 1.10 of \cite{inductivecounting} to prove the estimate for the number of fields having Galois group $G$. Observe that the action of $G$ is imprimitive as there are intermediate subextensions and $G$ has the tower type $(C_2, \, S_3)$ where $C_2 \leq S_2$ and $S_3 \leq S_6$ has transitive actions. Then, applying Theorems \ref{thm:class group bound} and \ref{thm:bound for S3 fields}, we have   
\begin{align}\label{eq:eq:4.3}
&\sum_{K \in \wt{\Q}{S_3}{X}{6}}\#\text{Hom}(\text{Cl}_K, C_2 ) =  \sum_{K \in \wt{\Q}{S_3}{X}{6}}\#\rcg{K}{2} 
\ll X^{\frac{3}{4}+\epsilon}.
\end{align}
Corollary 1.10 of \cite{inductivecounting} provides an upper bound for $\z{F}{G}{X}{m\#A}$ when the group $G$ is an imprimitive transitive permutation group and is a subgroup of $A \wr B$ for which $A$ is a finite abelian group and $B$ is a transitive permutation group of degree $m$ such that there exists at least one $B$ extension of $F.$ In particular, we can take $A=C_2$, $B=S_3$, and $\theta = \frac{3}{4}+\epsilon$ in the statement of Corollary 1.10 of \cite{inductivecounting}. If $a( G \cap C_2^6)\geq 3$ then $$\theta \geq \frac{3}{4}\geq \frac{2}{3} \geq \frac{\#A}{a( G\cap C_2^6)}.$$ In this case, by Corollary 1.10 of \cite{inductivecounting} we get 
 \begin{equation}\label{eqn5.3}
   \z{\Q}{G}{X}{12} \ll X^{\frac{3}{8}+\epsilon}.  
 \end{equation}
On the other hand, if $a(G \cap C_2^6) =2$ we have $\theta < \frac{\#A}{a( G \cap C_2^6)}.$ Then, we get 
\begin{equation}\label{eqn5.4}
    \z{\Q}{G}{X}{12} \ll X^{\frac{1}{2}}.
\end{equation}

Now let us assume $a(G \cap C_2^6)=1$. Then, we get $
\theta \leq 2.$ By Corollary 1.10 of \cite{inductivecounting}, we get 
$$\z{\Q}{G}{X}{12} \sim c(F, G) X (\log{X})^{b-1}$$
for some positive integer $b.$ However, Klüners proved in \cite{kluners} that $E_{12}(\Q, G, X) \ll X^{\frac{3}{4}+\epsilon}.$ This leads to a contradiction. Therefore, we get $a(G\cap C_2^6) \neq 1$ and hence, 
\begin{equation}\label{eq:bound subgroups C2 wreath S3}
    E_{12}(\Q, \, C_2 \wr S_3, \, X) \ll X^{\frac{1}{2}}.
\end{equation}

\begin{proof}[Proof of Theorem \ref{thm:main theorem S3}]
Applying Theorem \ref{thm:explicitC2} on the left-hand side of equation $\eqref{eq:count quadratic over S3}$, we get

\begin{align}\label{eq:separate main and error S3}
\sum_{K \in \mathcal{F}_6(\Q, \, S_3, \, X^{1/2})} \sum_{L \in \wt{K}{C_2}{X/|\ds{K}|^2}{2}}  1 = &\sum_{K \in \mathcal{F}_6(\Q, \, S_3, \, X^{1/2})} \frac{1}{2^{i(K)}}\cdot \frac{\mathrm{Res}_{s=1}\zeta_K(s)}{\zeta_K(2)}\cdot \frac{X}{|\ds{K}|^2} + \nonumber \\
&\sum_{K \in \mathcal{F}_6(\Q, \, S_3, \, X^{1/2})} \#\rcg{K}{2}\cdot  \O\left( |\ds{K}|^{\frac{1}{7}}\cdot\left(\frac{X}{|\ds{K}|^2}\right)^{\frac{5}{7}+\epsilon}\right).
\end{align}
Observe that by Theorem \ref{thm:bound for S3 fields} we have $a= \frac{1}{3}$ in this case. Thus, by equation \eqref{eqn:mainterm}, the main term is given as follows:

\begin{align}\label{eqn:maintermS3}
    &  \sum_{K \in \mathcal{F}_6(\Q, \, S_3, \, X^{1/2})} \frac{1}{2^{i(K)}}\cdot \frac{\mathrm{Res}_{s=1}\zeta_K(s)}{\zeta_K(2)}\cdot \frac{X}{|\ds{K}|^2} = c(\Q, C_2 \wr S_3)X+\O(X^{\frac{1}{6}+\epsilon}).
\end{align}
Furthermore, we have $\frac{1}{3}< \frac{3}{2} - \frac{9\cdot6-1}{2\cdot6(6+1)}$. Therefore, in this case, by equation \eqref{eq:bound first part of the error term}, the error term will be
\begin{align}\label{eqnerrs3}
\sum_{K \in \mathcal{F}_6(\Q, \, S_3, \, X^{1/2})} \#\rcg{K}{2}\cdot  |\ds{K}|^{\frac{1}{7}}\cdot\left(\frac{X}{|\ds{K}|^2}\right)^{\frac{5}{7}+\epsilon} \ll X^{\frac{5}{7}+\epsilon}  
\end{align}
Combining equations \eqref{eq:count quadratic over S3}, \eqref{eq:bound subgroups C2 wreath S3},  \eqref{eq:separate main and error S3}, \eqref{eqn:maintermS3}, and \eqref{eqnerrs3} we get 
$$\z{\Q}{C_2 \wr S_3}{X}{12}= c(\Q,C_2 \wr S_3)X+\O(X^{\frac{5}{7}+\epsilon}).$$
This completes the proof. 
\end{proof}

\section{Analyzing the case \texorpdfstring{$F= \mathbb{Q}, \, H \cong C_3$}{TEXT}}

This section aims to count the $C_2 \wr C_3$ fields of degree $6$ over the base field $\Q$, i.e, we wish to evaluate $\z{\Q}{C_2 \wr C_3}{X}{6}.$ First, let us restate equation \eqref{thm:count quadratic over H extension} for the specific case of $H \cong C_3.$ 

\begin{align}\label{eq:count C2 over C3}
\sum_{K \in \mathcal{F}_3(\Q, \, C_3, \, X^{1/2})} \sum_{L \in \wt{K}{C_2}{X/|\ds{K}|^2}{2}} 1 =  & \alpha(C_3)\z{\Q}{C_2 \wr C_3}{X}{6}  +E_{6}(F, C_2 \wr C_3, X)
\end{align}
First, let us evaluate the constant $\alpha(C_3)$. 
\begin{lem}
We have $\alpha(C_3)=1.$    
\end{lem}
\begin{proof}
Suppose $L/ \Q $ is an extension of degree $6$ with Galois group $\gal{\widetilde{L}/ \Q} \cong C_2 \wr C_3.$ Working towards a contradiction, let us assume that the left-hand side of equation \eqref{eq:count C2 over C3} counts $L/ \Q$ with a multiplicity of more than one. This is possible only if there exists at least two subextensions of the form  $L/K_i/ \Q$ for $i=1, \, 2$ with $K_i/ \Q$ being cubic $C_3$ extensions and $L/K_i$ are quadratic extensions. Then, by the Fundamental Theorem of Galois theory, $\gal{\widetilde{L}/K_1}$ and  $\gal{\widetilde{L}/K_2}$ are two distinct subgroups of $C_2 \wr C_3$ with cardinality $8.$
However, $C_2 \wr C_3 \cong C_2 \times A_4$ has only one subgroup of order 8. This leads to a contradiction. Therefore $\alpha(C_3)=1.$
\end{proof}
\subsection{Bound for \texorpdfstring{$E_6(\Q, \, C_2 \wr C_3, \, X)$}{TEXT}}

\parindent=10pt Now, let us analyze $E_6(\Q, \, C_2 \wr C_3 , \, X).$ Suppose $K/ \Q$ is a number field of degree $3$ and $L/ K$ is a field extension of degree $2.$ Let $\widetilde{L}$ be the Galois closure of $L/ \Q.$

Then, $ \gal{\widetilde{L} / \Q}$ is a transitive subgroup of $S_6$ and a subgroup of $C_2 \wr C_3 \cong C_2 \times A_4.$ If $\gal{\widetilde{L} / \Q}$ is a proper subgroup of $C_2 \wr C_3$, then the only possibilities are $\gal{\widetilde{L} / \Q} \cong C_6$ or $\gal{\widetilde{L} / \Q} \cong A_4.$ Using Theorem I.2 of \cite{wrightdensity}, we get the following upper bounds
\begin{equation}\label{eqn:boundC3}
\z{\Q}{C_3}{X}{3} \ll X^{\frac{1}{2}+\epsilon},    \end{equation}
and  
\begin{equation}\label{eqn6.6}
    \z{\Q}{C_6}{X}{6} \ll X^{\frac{1}{3}+\epsilon}.
\end{equation}

For $A_4$, the Malle's constant is given by $a(A_4)= \frac{1}{2}.$ Observe that $A_4 \leq S_6$ has an imprimitive action. Furthermore, $ A_4$ has a tower type $(C_2, \, C_3)$ where $C_2 \leq S_2$ and $C_3 \leq S_3$ has primitive action. $C_2$ is a finite abelian group and $C_2^3 \cap A_4 = V_4.$  We have 

\begin{align*}
\displaystyle    & \sum_{F \in \z{\Q}{C_3}{X}{3}}  \#\text{Hom}(\text{Cl}_F, \, C_2 )  = \sum_{F \in \z{\Q}{C_3}{X}{3}} \# \text{Cl}_F[2]  \ll  X^{\frac{1}{2}+ \delta +\epsilon}.
\end{align*}
Hence, we can choose $A=C_2$ and $\theta= \frac{1}{2}+ \delta+\epsilon$ in the statement of Corollary 1.10 of \cite{inductivecounting}. Furthermore, we have $a(V_4)= 2.$ Clearly, $\theta < \frac{\#C_2}{a(V_4)}=1$. Thus, by Corollary 1.10 of \cite{inductivecounting}, we get 

\begin{equation}\label{eqn6.7}
    \z{\Q}{A_4}{X}{6} \ll X^{\frac{1}{2}+\epsilon}.
\end{equation}
Combining equations \eqref{eqn6.6} and \eqref{eqn6.7}, we get 

\begin{equation}\label{eq:bound subgroups C2 wreath C3}
E_6(\Q, \, C_2 \wr C_3, \, X) \ll X^{\frac{1}{2}+\epsilon}.
\end{equation}

\subsection{Main term for \texorpdfstring{$\z{\Q}{C_2 \wr C_3}{X}{6}$}{Text}}

\parindent=10pt Now, let us analyse the left-hand side of equation \eqref{eq:count C2 over C3}Applying Theorem \ref{thm:explicitC2}, we get 

\begin{align}\label{eq:separate main and error term C3}
&\sum_{K \in \mathcal{F}_3(\Q,\, C_3, \,X^{1/2})} \sum_{L \in \wt{K}{C_2}{X/|\ds{K}|^2 }{2}} 1 \nonumber \\
=  &\sum_{K \in \mathcal{F}_3(\Q,\, C_3, \,X^{1/2})} \frac{1}{2^{i(K)}}\cdot \frac{\mathrm{Res}_{s=1}\zeta_K(s)}{\zeta_K(2)}\cdot \frac{X}{|\ds{K}|^2}  + \nonumber \\
&\sum_{K \in \mathcal{F}_3(\Q,\, C_3, \,X^{1/2})} \#\rcg{K}{2}\cdot \O\left( |\ds{K}|^{\frac{1}{4}}\cdot (\log(|\ds{K}|))^2\cdot \left(\frac{X}{|\ds{K}|^2}\right)^{\frac{1}{2}+\epsilon} \right).
\end{align}

Observe that in this case, by equation \eqref{eqn:boundC3} we can take $a=\frac{1}{2}$. Now, using equation \eqref{eqn:mainterm}, we conclude that the main term is given by
\begin{equation}\label{eq:main term C3}
\sum_{K \in \mathcal{F}_3(\Q,\, C_3, \,X^{1/2})} \frac{1}{2^{i(K)}}\cdot \frac{\mathrm{Res}_{s=1}\zeta_K(s)}{\zeta_K(2)}\cdot \frac{X}{|\ds{K}|^2} =c(\Q, C_2 \wr C_3)X+\O(X^{\frac{1}{4}+\epsilon}).
\end{equation}

\subsection{The error term in counting \texorpdfstring{$C_2 \wr C_3$}{TEXT} fields}

\parindent=10pt Now, let us focus on the error term in equation \eqref{eq:separate main and error term C3}. First, let us separate the sum into two parts.

\begin{equation*}
 \sum_{K \in \wt{\Q}{C_3}{Z}{3}} E + \sum_{K \in \wt{\Q}{C_3}{X^{1/2}}{3} \setminus \wt{\Q}{C_3}{Z}{3}} E  
\end{equation*}
where $E$ denotes the summand in the second term of equation \eqref{eq:separate main and error term C3}. Later, we will choose a suitable value for $Z$ to get the best possible error term. 

Suppose $\delta$ is a positive constant such that for all cubic fields $K$, we have

\begin{equation*}
    \#\rcg{K}{2}\ll |\ds{K}|^{\delta+\epsilon}.
\end{equation*}
Then, by using Lemma \ref{lemma: main lemma}, the first part in the above sum can be bounded as follows:

\begin{align}\label{eqn6.4}
&\sum_{K \in \wt{\Q}{C_3}{Z}{3}} \#\rcg{K}{2}\cdot  \O\left( |\ds{K}|^{\frac{1}{4}}\cdot (\log(| \ds{K}|))^2\cdot \left(\frac{X}{|\ds{K}|^2}\right)^{\frac{1}{2}+\epsilon}. \right) \nonumber \\
& \ll X^{\frac{1}{2}+\epsilon}\cdot \sum_{K \in \wt{\Q}{C_3}{Z}{3}} |\ds{K}|^{-\frac{3}{4}+\delta+\epsilon} \nonumber \\
& \ll X^{\frac{1}{2}+\epsilon}\cdot( Z^{-\frac{1}{4}+\delta+\epsilon}+1).
\end{align}
We remark that, by Theorem \ref{thm:class group bound} the best possible  value for $\delta$ is $0.2784\dots.$ Thus, the dominating term in equation \eqref{eqn6.4} is $X^{\frac{1}{2}+\epsilon}Z^{-\frac{1}{4}+\delta+\epsilon}.$

Next, to bound the second part of the error term, we use the following weaker bound 

\begin{align}\label{eqn6.5}
& \sum_{K \in \mathcal{F}_3(\Q,\, C_3, \,X^{1/2}) \setminus\mathcal{F}_3(\Q,\, C_3, \,Z)} \sum_{L \in \wt{K}{C_2}{X/|\ds{K}|^2}{2}} 1  \nonumber \\ 
\leq  &\sum_{K \in \mathcal{F}_3(\Q,\, C_3, \,X^{1/2}) \setminus\mathcal{F}_3(\Q,\, C_3, \,Z)} \sum_{\substack{\mathfrak{a} \in A(K) \\ |\mathrm{Nm}_{K/\Q}(\mathfrak{a})|\leq  X/|\ds{K}|^2}} \sum_{\overline{u} \in S_2(K)} 1 \nonumber \\
\ll & \sum_{K \in \mathcal{F}_3(\Q,\, C_3, \,X^{1/2}) \setminus\mathcal{F}_3(\Q,\, C_3, \,Z)} |\ds{K}|^{\delta +\epsilon} \, \frac{X}{\ds{K}|^2}  \nonumber \\
\ll & X Z^{-\frac{3}{2}+\delta+\epsilon}.
\end{align}
By comparing equations \eqref{eq:separate main and error term C3} and \eqref{eqn6.5}, we get

\begin{align}\label{eq:bound second part of error term C3}
    \sum_{K \in \mathcal{F}_3(\Q,\, C_3, \,X^{1/2}) \setminus \wt{\Q}{C_3}{Z}{3}} \#\rcg{K}{2}\cdot  |\ds{K}|^{\frac{1}{4}}\cdot (\log(|\ds{K}|))^2\cdot \left(\frac{X}{|\ds{K}|^2}\right)^{\frac{1}{2}+\epsilon} \ll X Z^{-\frac{3}{2}+\delta+\epsilon}.
\end{align}

\begin{proof}[Proof of Theorem \ref{thm:main theorem C3} ]
Combining equations \eqref{eqn6.4} and \eqref{eq:bound second part of error term C3}, we get 

\begin{align}
& \sum_{K \in \mathcal{F}_3(\Q,\, C_3, \,X^{1/2})} \#\rcg{K}{2}\cdot  |\ds{K}|^{\frac{1}{4}}\cdot (\log(|\ds{K}|))^2\cdot \left(\frac{X}{|\ds{K}|^2}\right)^{\frac{1}{2}+\epsilon} \nonumber \\
\ll   &X^{\frac{1}{2}+\epsilon}\cdot Z^{-\frac{1}{4}+\delta+\epsilon}+X. Z^{-\frac{3}{2}+\delta+\epsilon}
\end{align}
This is optimised when $Z=X^{\frac{2}{5}}$, which gives us
\begin{align}\label{eq:errC3}
& \sum_{K \in \mathcal{F}_3(\Q,\, C_3, \,X^{1/2})} \#\rcg{K}{2}\cdot  |\ds{K}|^{\frac{1}{4}}\cdot (\log(|\ds{K}|))^2\cdot \left(\frac{X}{|\ds{K}|^2}\right)^{\frac{1}{2}+\epsilon} \ll X^{\frac{2\left(1+\delta \right)}{5}+\epsilon} .
\end{align}
Finally, combining equations \eqref{eq:count C2 over C3}, \eqref{eq:bound subgroups C2 wreath C3}, \eqref{eq:separate main and error term C3}, \eqref{eq:main term C3}, and \eqref{eq:errC3}, we get
\begin{equation*}
    \z{\Q}{C_2 \wr C_3}{X}{6}=c(\Q, C_2 \wr C_3)X+\O(X^{\frac{2\left(1+\delta \right)}{5}+\epsilon}).
\end{equation*}
This completes the proof of Theorem \ref{thm:main theorem C3}.
\end{proof}

\section*{Acknowledgment}
I am grateful to my advisor, Professor Alina Bucur, for her guidance and support throughout this project. I also thank Professor Brandon Alberts for valuable discussions and for his comments on an earlier draft of this paper. I further thank Professors Amanda Tucker and Kevin J. McGown for their helpful suggestions and assistance.

\bibliographystyle{plain}
\bibliography{biblio.bib}

@article{kevinamanda,
author = {McGown, Kevin J. and Tucker, Amanda},
year = {2024},
month = {06},
pages = {},
title = {An improved error term for counting {$D_4$}‐quartic fields},
volume = {56},
journal = {Bulletin of the London Mathematical Society},
doi = {10.1112/blms.13106}
}

@article{kluners,
author = {Klüners, Jürgen},
title = {THE DISTRIBUTION OF NUMBER FIELDS WITH WREATH PRODUCTS AS {G}ALOIS GROUPS},
journal = {International Journal of Number Theory},
volume = {08},
number = {03},
pages = {845-858},
year = {2012},
doi = {10.1142/S1793042112500492},

URL = { 
    
        https://doi.org/10.1142/S1793042112500492
    
    

},
eprint = { 
    
        https://doi.org/10.1142/S1793042112500492
    
    

}
}

@article{bhar,
author = {Bhargava, Manjul and Wood, Melanie Matchett},
year = {2007},
month = {05},
pages = {1581-1588},
title = {The density of discriminants of {$S_3$}-sextic number fields},
volume = {136},
journal = {Proceedings of The American Mathematical Society},
doi = {10.1090/S0002-9939-07-09171-X}
}

@article{bhargava2017bounds2torsionclassgroups,
  title={Bounds on 2-torsion in class groups of number fields and integral points on elliptic curves},
  author={Bhargava, Manjul and Shankar, Arul and Taniguchi, Takashi and Thorne, Frank and Tsimerman, Jacob and Zhao, Yongqiang},
  journal={Journal of the American Mathematical Society},
  volume={33},
  number={4},
  pages={1087--1099},
  year={2020}
}

@misc{inductivecounting,
      title={Inductive Methods For Counting Number Fields},
      author={Brandon Alberts and Robert J. Lemke Oliver and Jiuya Wang and Melanie Matchett Wood},
      year={2024}
}

@article{cohend4,
      title={Enumerating Quartic Dihedral Extensions of $\mathbb{Q}$},
      author={ Henri Cohen and Francisco Diaz Y Diaz and Michel Olivier },
      year={2002},
      journal={Compositio Mathematica 133: 65–93},

}

@incollection{alinad4,
  title={Power-Saving Error Terms for the Number of {$D_4$}-Quartic Extensions over a Number Field Ordered by Discriminant},
  author={Bucur, Alina and Florea, Alexandra and L{\'o}pez, Allechar Serrano and Varma, Ila},
  booktitle={Research Directions in Number Theory: Women in Numbers V},
  pages={197--218},
  year={2024},
  publisher={Springer}
}

@article{wrightdensity,
     author = {Wright, David J.},
     title = {Distribution of Discriminants of Abelian Extensions},
     journal = {Proceedings of the London Mathematical Society},
     volume = {s3-58},
     number = {1},
     pages = {17-50},
     year = {1989}
}

@article{lowryduda2018uniformboundslatticepoint,
  title={Uniform bounds for lattice point counting and partial sums of zeta functions},
  author={Lowry-Duda, David and Taniguchi, Takashi and Thorne, Frank},
  journal={Mathematische Zeitschrift},
  pages={1--20},
  year={2022},
  publisher={Springer}
}

@article{Malle1,
    author={Gunter Malle}, 
    title = {On the distribution of {G}alois groups },
    journal ={Journal of Number Theory},
    year = {2002},
    volume= {92:315–322}
}

@article{Malle2,
    author={Gunter Malle}, 
    title = {On the distribution of {G}alois groups {II}},
    journal ={Exp. Math.},
    year = {2004},
    volume= {13:129–135}
}

@article{kluners2005counter,
  title={A counter example to {M}alle's conjecture on the asymptotics of discriminants},
  author={Kl{\"u}ners, J{\"u}rgen},
  journal={Comptes Rendus Mathematique},
  volume={340},
  number={6},
  pages={411--414},
  year={2005},
  publisher={Elsevier}
}

\end{document}